\documentclass[reqno, 11pt]{amsart}
\usepackage{placeins}

\usepackage[utf8]{inputenc}
\usepackage[T1]{fontenc}

\usepackage{amsfonts,amssymb,enumerate,amsthm,amsmath,color,graphicx,mathtools}
\usepackage{hyperref}
\usepackage[all]{xy}

\usepackage[a4paper,margin=4cm,includehead,includefoot,headsep=12pt,footskip=18pt]{geometry}

\newtheorem{theor}{Theorem}[section]
\newtheorem{theorem}[theor]{Theorem}
\newtheorem{prop}[theor]{Proposition}

\newtheorem{lem}[theor]{Lemma}
\newtheorem{lemma}[theor]{Lemma}

\newtheorem{remark}{Remark}

\newcommand{\B}{\mathbb{B}}

\newcommand{\C}{\mathbb{C}}

\newcommand{\rank}{\operatorname {rank}}

\newcommand{\W}{\Omega}

\newcommand{\ov}[1]{\overline{#1}}

\newcommand{\K}{K\"ahler}

\newcommand{\Aut}{\operatorname {Aut}}

\newcommand{\Span}{\operatorname {span}}

\begin{document}

\title[The polydisk theorem for Hartogs domains]{The polydisk theorem for Hartogs domains over  symmetric domains}

\author{Andrea Loi}
\address{(Andrea Loi) Dipartimento di Matematica \\
         Universit\`a di Cagliari (Italy)}
         \email{loi@unica.it}

\author{Roberto Mossa}
\address{(Roberto Mossa) Dipartimento di Matematica \\
         Universit\`a di Cagliari (Italy)}
        \email{roberto.mossa@unica.it}
        
        \author{Fabio Zuddas}
\address{(Fabio Zuddas) Dipartimento di Matematica \\
         Universit\`a di Cagliari (Italy)}
        \email{fabio.zuddas@unica.it}

\thanks{
The authors are supported by INdAM and  GNSAGA - Gruppo Nazionale per le Strutture Algebriche, Geometriche e le loro Applicazioni, by GOACT - Funded by Fondazione di Sardegna and 
partially funded by PNRR e.INS Ecosystem of Innovation for Next Generation Sardinia (CUP F53C22000430001, codice MUR ECS00000038).}


\begin{abstract}
We extend the polydisk theorem of \cite{MOSSAZEDDA2022polch}, originally established for classical Cartan-Hartogs domains, to Hartogs domains over arbitrary (possibly reducible and exceptional) bounded symmetric domains. We further establish a dual counterpart of this result. As an application, we show that the dual of a Hartogs domain over a bounded symmetric domain admits no totally geodesic immersion into any compact Riemannian manifold, thereby broadening the rigidity phenomena obtained in \cite{LOIMOSSAZUDDAS2024dualemb}.\end{abstract}
 
\maketitle

\tableofcontents
\section{Introduction}

Let $\Omega\subset\mathbb{C}^n$ be a bounded symmetric domain endowed with its hyperbolic metric $g_{\Omega}$, whose associated \K\  form is
$$
\omega_{\Omega}
\;=\;
-\tfrac{i}{2}\,\partial\bar\partial \log N_{\Omega}(z,\bar z),
$$
where $N_{\Omega}(z,\bar z)$ denotes the generic norm of the Hermitian positive Jordan triple system associated to $\Omega$ (see below).
The \emph{Hartogs domain over $\Omega$} is defined by
\[
M_{\Omega,\mu}
=\bigl\{(z_{0},z)\in\mathbb{C}\times\Omega
\;\big|\;
|z_{0}|^{2}<N_{\Omega}^\mu(z,\bar z)\bigr\}.
\]
This is also known as a {\em generalized Cartan--Hartogs domain}. 
A Cartan--Hartogs domain is a Hartogs domain whose base $\Omega$ is irreducible (i.e., a Cartan domain). 
It is known that $M_{\Omega,\mu}$ is homogeneous if and only if it is the unit ball $\B^{n+1}$, 
equivalently when $\Omega=\B^n$ and $\mu=1$.

On $M_{\Omega,\mu}$ one considers the \K\   metric $g_{\Omega, \mu}$ whose associated \K\ form is given by:
\begin{equation}\label{roosmetric}
\omega_{\Omega, \mu}
=-\frac{i}{2}\,\partial\bar\partial
\log\ \!\Bigl(N_{\Omega}^\mu(z,\bar z)-|z_{0}|^{2}\Bigr).
\end{equation}

The \emph{dual Hartogs domain over $\Omega$} is $\mathbb{C}^{n+1}$ equipped with the \emph{dual \K\ metric} whose associated \K\  form is:
\begin{equation}\label{roosmetricdual}
\omega^{*}_{\Omega, \mu}
=\frac{i}{2}\,\partial\bar\partial
\log\ \!\Bigl(N_{\Omega}^\mu(z,-\bar z)+|z_{0}|^{2}\Bigr),
\end{equation}
and the dual bounded symmetric domain is $\C^n$
with the dual \K\ metric $g_{\Omega}^*$ whose associated
\K\ form is
$$
\omega^*_{\Omega}
\;=\;
\tfrac{i}{2}\,\partial\bar\partial \log N_{\Omega}(z,-\bar z).
$$
Hartogs domains over bounded symmetric domains, equipped with the \K\ metric~\eqref{roosmetric} and its dual~\eqref{roosmetricdual}, form a tractable yet genuinely nonhomogeneous class of \K\ manifolds; 
In this framework, \emph{balanced metrics} (in Donaldson's sense \cite{Donaldson2001}) are completely understood: on Cartan domains they are fully classified, and among Cartan--Hartogs domains the complex hyperbolic space is the only one admitting a balanced metric. 
Beyond this, the family offers a convenient testing ground for explicit K\"ahler--Einstein geometry: existence and closed formulas for the K\"ahler--Einstein metric are available \cite{WangWangZhang2009,DengXiaoYang2011}, moreover  the K\"ahler--Einstein and Bergman metrics are equivalent \cite{YinZhang2008,ZhaoLin2008}, and canonical metrics on Cartan--Hartogs domains have been analyzed in \cite{Zedda2012}. 
Finally, holomorphic isometric immersions into the infinite-dimensional complex projective space $\mathbb{C}P^\infty$ are constructed, yielding the first examples  of  nonhomogeneous and complete K\"ahler--Einstein
submanifolds of $\mathbb{C}P^\infty$\cite{LoiZedda2011,HaoWangZhang2015,LoiZeddaBook2018}.

\vskip 0.3cm

We now recall two further definitions needed for our main results: the \emph{Hartogs-Polydisk} and its dual.  First, the \emph{polydisk of dimension $r$} is
\[
\Delta^{r}
=\bigl\{\,z=(z_{1},\dots,z_{r})\in\mathbb{C}^{r}
\;\big|\;
|z_{j}|^{2}<1,\;j=1,\dots,r\bigr\},
\]
whose generic norm is
\[
N_{\Delta^{r}}(z,\bar z)
=\prod_{j=1}^{r}\bigl(1-|z_{j}|^{2}\bigr).
\]
Hence the \K\ metric is the product of hyperbolic metrics and its associated \K\  form is
\[
\omega_{\Delta^{r}}
=-\frac{i}{2}\,\partial\bar\partial
\log\!\prod_{j=1}^{r}(1-|z_{j}|^{2}).
\]
In particular, $(\Delta^{r},\omega_{\Delta^{r}})$ is homorphically isometric to $(\C H^1,\omega_{\C H^1})^r$.
Setting $\Omega=\Delta^{r}$ in the above definitions yields the \emph{Hartogs–Polydisk}:
\[
M_{\Delta^{r},\mu}
=\bigl\{(z_{0},z)\in\mathbb{C}\times\Delta^{r}
\;\big|\;
|z_{0}|^{2}<\prod_{j=1}^{r}(1-|z_{j}|^{2})^{\mu}\bigr\},
\]
with associated Kähler form
\[
\omega_{\Delta^{r},\mu}
=-\frac{i}{2}\,\partial\bar\partial
\log\!\Bigl(\prod_{j=1}^{r}(1-|z_{j}|^{2})^{\mu}-|z_{0}|^{2}\Bigr),
\]

Finally, the \emph{dual Hartogs–Polydisk} is $\mathbb{C}^{r+1}$  endowed with the \K\ metric whose associated \K\ form is
\[
\omega^{*}_{\Delta^{r},\mu}
=\frac{i}{2}\,\partial\bar\partial
\log\!\Bigl(\prod_{j=1}^{r}(1+|z_{j}|^{2})^{\mu}+|z_{0}|^{2}\Bigr),
\]

We present now the two  main theorems established in this paper, which provide essential structural results about totally geodesic embeddings of polydisks into Hartogs domains over bounded symmetric domains and their duals.

\begin{theorem}\label{THMhpoly}
Let \(\Omega\subset\mathbb{C}^n\) be a bounded symmetric domain of rank \(r\).
Then for every point \(p\in M_{\Omega,\mu}\) and every vector \(X\in T_{p}M_{\Omega,\mu}\),
there exist  totally geodesic holomorphic  embeddings
\[j \colon \Delta^{r}\hookrightarrow\Omega, \  \widetilde{j}\colon M_{\Delta^{r},\mu}\longrightarrow M_{\Omega,\mu}\] 
such that  $\widetilde{j}$ is an extension of $j$, 
\(\widetilde{j}(q)=p\) and
\((d\widetilde{j})_{q}(V)=X\) for some \(q\in M_{\Delta^{r},\mu}\) and \(V\in T_qM_{\Delta^{r},\mu}\).
\end{theorem}

\begin{theorem}\label{THMhpolydual}
Let $\Omega \subset \mathbb{C}^n$ be a bounded symmetric domain of rank $r$.
For every $X \in T_{(z_0, 0)}\mathbb{C}^{\,n+1}$, there exists a  totally geodesic holomorphic embedding
\[
j: \bigl(\mathbb{C}^r,\, g_{\Delta^r}^*\bigr)\hookrightarrow \bigl(\mathbb{C}^n,\, g_{\Omega}^*\bigr)
\]
with $j(0)=0$ such that  the induced holomorphic embedding
\[
\widetilde{j}\colon \mathbb{C}^{r+1}\longrightarrow \mathbb{C}^{n+1},\qquad
\widetilde{j}(z_0,z)=(z_0, j(z)),
\]
is totally geodesic and
$
(d\widetilde{j})_{(z_0,0)}(V_0)=X,
$
for some $V_0\in T_{(z_0,0)}\mathbb{C}^{\,r+1}$. 
\end{theorem}

Theorem \ref{THMhpoly} extends the main result of \cite{MOSSAZEDDA2022polch}, originally proved for classical Cartan-Hartogs domains, to Hartogs domains over arbitrary bounded symmetric domains. Our result should  be viewed as a counterpart of the standard polydisk theorem for Hermitian symmetric spaces of noncompact types (see \cite[Part~1]{WolfBook1969FineStructureHSSpolydisc}). Dually, Theorem \ref{THMhpolydual} provides the affine-chart version of the polydisk theorem for Hermitian symmetric spaces of compact type (see \cite{HELGASONbookDiffGeom1978}). 
The polydisk theorem has numerous applications  including rigidity phenomena \cite{MOK1989MetricRigidityThmLocHSS, Mok2011ALM} and it is fundamental  for the study of  the symplectic geometry of Hermitian symmetric spaces of noncompact type
and their duals \cite{DISCALALOI2008sympdual,LoiMossaZuddas2015}.
Moreover, it has proved useful in the study of the diastatic exponential \cite{LOIMOSSA2011DistExp} and in the analysis of the volume and diastatic entropy of symmetric bounded domains \cite{MOSSA2013VolumeEntr}. We therefore expect our results to serve as an effective tool for addressing geometric problems on Hartogs domains and on their duals.

\smallskip

For example, as a main consequence of our analysis, we obtain the following obstruction.

\begin{theorem}\label{THMCORhpolydual}
Let $(\mathbb{C}^{n+1}, g_{{\Omega,\mu}}^{*})$ be a dual Hartogs domain. If $\mu \neq 1$ or $\operatorname{rank}(\Omega) \ge 2$, then there does not exist a totally geodesic isometric immersion
\[
f: \big(\mathbb{C}^{n+1}, g_{{\Omega,\mu}}^{*}\big) \hookrightarrow (M,g)
\]
into any compact Riemannian manifold $(M,g)$ with $\dim M \ge n+1$.
\end{theorem}

This theorem significantly sharpens the obstruction proved in \cite{LOIMOSSAZUDDAS2024dualemb}: 
apart from the homogeneous ball case $\Omega=\B^{n}$ and $\mu=1$, 
 there is no compact K\"ahler manifold $(N,h)$ of complex dimension $n+1$ containing 
$(\C^{n+1}, g_{{\Omega,\mu}}^{*})$ as an open dense subset.

\smallskip

\begin{remark}\rm
The same techniques developed in this paper show that Theorems \ref{THMhpoly} and \ref{THMhpolydual} remain valid for the Bergman metric $g_{M_{\Omega, \mu}}$ and for the \K\ metric $\hat g_{\Omega,\mu}$ on $M_{\Omega,\mu}$ (introduced in \cite{LOIMOSSAZUDDAS2024dualemb}), together with their duals, namely $g_{M_{\Omega, \mu}}^{*}$ and $\hat g^{*}_{\Omega,\mu}$ (see \cite{LOIMOSSAZUDDAS2025bergman}). By contrast, whether Theorem \ref{THMCORhpolydual} holds for the Bergman metric is still open; see \cite[Sect.~4]{LOIMOSSAZUDDAS2025bergman} for a discussion.
\end{remark}

The paper is organized as follows.
Section~\ref{SECcartan} reviews Hermitian symmetric spaces of noncompact type in their realization as bounded symmetric domains and constructs explicit totally geodesic polydisk embeddings, including the exceptional cases (Propositions~\ref{PROPpolymaxVI} and \ref{PROPpolymaxV}). In the following section we prove our main results, Theorems~\ref{THMhpoly} and \ref{THMhpolydual}, and derive the compactification obstruction in Theorem~\ref{THMCORhpolydual}. Along the way we record the Hartogs-polydisk reduction and the lifting of base automorphisms (Propositions~\ref{PROPPolyDisk1}, \ref{PROPPolyDisk2} and Lemmas~\ref{LEMliftgen}, \ref{LEMliftgend}), which will be used repeatedly in the proofs.

\section{Hermitian Symmetric Spaces  and Polydisk Embeddings}\label{SECcartan}
In this section, we begin by recalling the classification of Hermitian symmetric spaces of noncompact type (HSSNCT) in their realization as bounded symmetric domains.  We then undertake an explicit investigation of the Hermitian positive Jordan triple–system (HPJTS) structures associated with the exceptional Cartan domains, and we present concrete, totally geodesic embeddings of maximal–rank polydisks, thereby extending the results previously obtained for the classical Cartan domains in \cite{MOSSAZEDDA2022polch}.  These constructions will play a central role in our subsequent analysis of the geometric and curvature properties of the Hartogs domains over bounded symmetric domains.

\vskip 0.2cm

Let $\Omega$ be a Hermitian symmetric space of noncompact type (HSSNCT) in its realization as bounded symmetric domain \((\Omega, c g_B)\), where \(\Omega \subset \mathbb{C}^n\) is a convex, circular domain, that is, \(z \in \Omega, \theta \in \mathbb{R} \Rightarrow e^{i\theta} z \in \Omega\) (see \cite{KOBAYASHI1959GeomBounded} for details) and $g_B$ denotes its Bergman metric. Every bounded symmetric domain can be decomposed as a product of irreducible factors, called \emph{Cartan domains}. According to É. Cartan’s classification, Cartan domains fall into two categories: classical and exceptional (see \cite{KOBAYASHIbookHyperbMfd1970} for details).

\subsection{Cartan domains}
Classical domains can be described in terms of complex matrices as follows:
$$
\begin{aligned}
& \Omega_I[n, m] = \left\{ Z \in M_{n, m}(\mathbb{C}) \mid I_n - Z Z^* > 0, \, n \leq m \right\}, \\
& \Omega_{II}[n] = \left\{ Z \in M_n(\mathbb{C}) \mid Z = -Z^T,\, I_n - Z Z^* > 0, \, n \geq 5 \right\}, \\
& \Omega_{III}[n] = \left\{ Z \in M_n(\mathbb{C}) \mid Z = Z^T,\, I_n - Z Z^* > 0, \, n \geq 2 \right\}, \\
& \Omega_{IV}[n] = \left\{ (z_1, \ldots, z_n) \in \mathbb{C}^n \,\middle|\, \sum_{j=1}^n |z_j|^2 < 1,\ 1 + \left| \sum_{j=1}^n z_j^2 \right|^2 - 2 \sum_{j=1}^n |z_j|^2 > 0,\ n \geq 5 \right\},
\end{aligned}
$$
where \( A > 0 \) means that \( A \) is positive definite.
The associated Bergman kernel functions are given by:
$$
\begin{aligned}
\mathrm{K}_{\Omega_I}(z, z) &= \frac{1}{V(\Omega_I)} \left[ \det(I_n - Z Z^*) \right]^{-(n + m)}, \\
\mathrm{K}_{\Omega_{II}}(z, z) &= \frac{1}{V(\Omega_{II})} \left[ \det(I_n - Z Z^*) \right]^{-(n - 1)}, \\
\mathrm{K}_{\Omega_{III}}(z, z) &= \frac{1}{V(\Omega_{III})} \left[ \det(I_n - Z Z^*) \right]^{-(n + 1)}, \\
\mathrm{K}_{\Omega_{IV}}(z, z) &= \frac{1}{V(\Omega_{IV})} \left( 1 + \left| \sum_{j=1}^n z_j^2 \right|^2 - 2 \sum_{j=1}^n |z_j|^2 \right)^{-n},
\end{aligned}
$$
where \( V(\Omega_j) \), \( j = I, \ldots, IV \), denotes the total volume of \( \Omega_j\) with respect to the Euclidean measure of the ambient complex Euclidean space (see \cite{DISCALALOI2007KmapHSSintoCP} for details). We refer the reader to \cite{WANGYINZHANGROSS2006KEonHartogs} for a more complete description of these domains.
sWe equip each domain \( \Omega \) with the hyperbolic metric \( g_{\Omega} \), a rescaling of the Bergman metric, characterized by having holomorphic sectional curvature with infimum equal to \(-4\). More precisely, the associated Kähler form is given by
$$
\omega_{\Omega} = -\frac{i}{2} \partial \bar{\partial} \log\left( N_{\Omega}(z, \bar{z}) \right),
$$
where
$$
N_{\Omega} = \left( V(\Omega)\, \mathrm{K}_{\Omega} \right)^{-\gamma}
$$
is the generic norm associated to \( \Omega \), $K_\Omega$ the Bergman kernel, and \( \gamma \) is the \emph{genus}, a numerical invariant of the domain.
The table below summarizes the numerical invariants associated to each Cartan domain.\begin{table}[h!]
    \centering
    \renewcommand{\arraystretch}{1.5}
    \begin{tabular}{|c|c|c|c|c|c|c|}
        \hline
         & {$\begin{array}{l}
        \Omega_I[n, m] \\[-0.2cm]
        \left\{n \leq m\right\}
        \end{array}$} & {$\begin{array}{l}
       \Omega_{II}[n] \\[-0.2cm]
        \left\{5 \leq n\right\}
        \end{array}$} & {$\begin{array}{l}
       \Omega_{III}[n] \\[-0.2cm]
        \left\{2 \leq n\right\}
        \end{array}$} &  {$\begin{array}{l}
       \Omega_{IV}[n] \\[-0.2cm]
        \left\{5 \leq n\right\}
        \end{array}$} & $\Omega_V$ & $\Omega_{VI}$ \\
        \hline
        $d$ & $nm$ & $\frac{(n-1)n}{2}$ & $\frac{(n+1)n}{2}$ & $n$ & $16$ & $27$ \\
        \hline
        $r$ & $n$ & $\left\lfloor \frac{n}{2} \right\rfloor$ & $n$ & $2$ & $2$ & $3$ \\
        \hline
        $a$ & {$\begin{array}{l}
        2, \text{ if } 2 \leq n \\[-0.2cm]
        0, \text{ if } n = 1
        \end{array}$} & $4$ & $1$ & $n-2$ & $6$ & $8$ \\
         \hline
        $b$ & $m-n$ &{$\begin{array}{l}
        0, \text{ if } n \text{ is even} \\[-0.2cm]
        2, \text{ if } n \text{ is odd }
        \end{array}$} & $0$ & $0$ & $4$ & $0$ \\
        \hline
        $\gamma$ & $m+n$ & $2n-2$ & $n+1$ & $n$ & $12$ & $18$ \\
        \hline
    \end{tabular}
\end{table}
\FloatBarrier
A Cartan domain \( \Omega \subset \mathbb{C}^n \) is uniquely determined by a triple of integers \( (r, a, b) \), where \( r \) denotes the rank of \( \Omega \), and \( a \), \( b \) are positive integers. The complex dimension \( d \) of \( \Omega \) satisfies \( 2d = r(2b + 2 + a(r - 1)) \), and the genus \( \gamma \) of \( \Omega \) is given by \( \gamma = (r - 1)a + b + 2 \). Observe that \( (\Omega, \omega_{\Omega}) = (\mathbb{C}H^n, g_{\text{hyp}}) \) if and only if the rank \( r = 1 \).
For a more detailed description of these invariants, which is not necessary for our purposes, see for example \cite{ARAZYbookAsurveyOfInvariant}, \cite{ZHANG1997ConstantOfBSD}.


\subsection{Hermitian positive Jordan triple system of type~VI}

Let $\mathbb{O}$ be the real algebra of octonions, and let $\mathbb{O}_{\mathbb{C}}:=\mathbb{O}\otimes_{\mathbb{R}}\mathbb{C}$ denote its complexification (see \cite[§3]{BAEZoctonions2002}). Every element $Z\in\mathbb{O}_{\mathbb{C}}$ can be written as
\begin{equation}\label{Zgrande}
Z=z_0+\mathbf z
=
z_0+\sum_{j=1}^{7}z_j e_j,
\qquad z_0,z_j\in\mathbb{C},
\end{equation}
with respect to the standard Cayley basis $\{1,e_1,\dots,e_7\}$. For $W=w_0+\mathbf w\in\mathbb{O}_{\mathbb{C}}$ we set
\[
\widetilde Z:=z_0-\mathbf z
\quad\text{(octonionic/Cayley conjugation)},\qquad
\overline Z:=\overline{z_0}+\overline{\mathbf z}
\quad\text{(complex conjugation)},
\]
and define the (complex-bilinear) octonionic product and bilinear form by
\[
ZW=\bigl(z_0w_0-\langle\mathbf z,\mathbf w\rangle,\;\;
z_0\mathbf w+w_0\mathbf z+\mathbf z\times\mathbf w\bigr),\qquad
\langle Z,W\rangle=z_0w_0+\langle\mathbf z,\mathbf w\rangle,
\]
where $\langle\mathbf z,\mathbf w\rangle:=\sum_{j=1}^7 z_j w_j$ and $\times$ denotes the complex-bilinear extension of the Cayley cross product on $\mathbb{C}^7$.

\medskip

Consider the complex vector space
\begin{equation}\label{HZ123}
\mathcal H=\bigl\{(z_1,z_2,z_3,Z_1,Z_2,Z_3)\mid z_j\in\mathbb C,\; Z_j\in\mathbb{O}_{\mathbb C},\;j=1,2,3\bigr\}\cong\mathbb C^{27}.
\end{equation}
We endow $\mathcal H$ with the Hermitian form
\[
(x\mid y)=\sum_{j=1}^{3} z_j\,\overline{w}_j+\sum_{j=1}^{3}\langle Z_j,\overline{W}_j\rangle,
\]
and the \emph{Freudenthal product}
\[
x\times y
=
\begin{pmatrix}
z_2w_3+z_3w_2-\langle Z_1,W_1\rangle\\[3pt]
z_3w_1+z_1w_3-\langle Z_2,W_2\rangle\\[3pt]
z_1w_2+z_2w_1-\langle Z_3,W_3\rangle\\[3pt]
Z_2W_3+W_2Z_3 - z_1\widetilde W_1 - w_1\widetilde Z_1\\[3pt]
Z_3W_1+W_3Z_1 - z_2\widetilde W_2 - w_2\widetilde Z_2\\[3pt]
Z_1W_2+W_1Z_2 - z_3\widetilde W_3 - w_3\widetilde Z_3
\end{pmatrix}^{\!\top}.
\]
The \emph{Jordan triple product} is
\begin{equation}\label{EQtriplPrVI}
\{x,y,z\}=(x\mid y)\,z+(z\mid y)\,x-(x\times z)\times\overline y.
\end{equation}
With these structures, $(\mathcal H,\{\cdot,\cdot,\cdot\})$ is the Hermitian positive Jordan triple system (HPJTS) of type VI. The {\em Freudenthal adjoint}
\[
z^{\sharp}:=\tfrac12\,z\times z,
\]
has the explicit expression

\begin{equation}\label{EQadjVI}\begin{split} 
z^{\sharp}=\Bigl(&
z_2 z_3 - \tfrac{1}{2}\,\langle Z_1,Z_1\rangle,\;
z_3 z_1 - \tfrac{1}{2}\,\langle Z_2,Z_2\rangle,\;
z_1 z_2 - \tfrac{1}{2}\,\langle Z_3,Z_3\rangle,\;\\
&Z_2 Z_3 - z_1 \widetilde{Z}_1,\;
Z_3 Z_1 - z_2 \widetilde{Z}_2,\;
Z_1 Z_2 - z_3 \widetilde{Z}_3
\Bigr). 
\end{split}\end{equation}

\subsection{The exceptional Cartan domain of type VI}
We refer the reader to \cite{RoosBook2000Exceptional, VivianiBook2014HSS} for details.
The exceptional bounded symmetric domain of type VI is the subset $\Omega_{\mathrm{VI}}\subset\mathcal H$ of complex dimension $27$ given in its circular realization by
\begin{equation}\label{EQeqVI}
\begin{gathered}
1-(z \mid  z)+(z^{\sharp}  \mid  {z^{\sharp}})-\left|\dfrac{1}{3}\,(z^{\sharp}\mid\ov z)\right|^2\;>\;0,\\[4pt]
3-2\,(z \mid z)+(z^{\sharp}\mid{z^{\sharp}})\;>\;0,\\[2pt]
3-(z \mid z)\;>\;0,
\end{gathered}
\end{equation}

The associated generic norm is
\begin{equation}\label{EQgennormVI}
N_{\Omega_{\mathrm{VI}}}(z,\ov w)
=
1-(z\mid w)+(z^{\sharp}\mid w^{\sharp})-
\frac{1}{9}\,(z^{\sharp}\mid\ov z)\,{(\ov w^{\sharp}\mid w)},
\qquad z,w\in\mathcal H.
\end{equation}


\vskip0.2cm

We now state and prove the following proposition, in which we exhibit explicitly a totally geodesic embedding of a maximal–rank polydisk $\Delta^3$ into the exceptional domain $\Omega_{\mathrm{VI}}$ with respect to the hyperbolic metric.

\begin{prop}\label{PROPpolymaxVI}
The holomorphic map
\[
f \colon \Delta^3 \longrightarrow \Omega_{\mathrm{VI}},
\qquad
f(z_{1},z_{2},z_{3})
=\bigl(z_{1},\,z_{2},\,z_{3},\,0,\,0,\,0\bigr),
\]
is an isometric embedding, i.e.
\[
g_{\Delta^3}=f^{*}g_{\Omega_{\mathrm{VI}}},
\]
and is totally geodesic. Consequently,
\[
\Pi \;=\;\operatorname{Im}(f)
=\bigl\{(z_{1},z_{2},z_{3},Z_{1},Z_{2},Z_{3})\in\Omega_{\mathrm{VI}}
\ \big|\ Z_{1}=Z_{2}=Z_{3}=0\bigr\}
\]
is a totally geodesic polydisk of maximal rank in $\Omega_{\mathrm{VI}}$.
\end{prop}

\begin{proof}
From \eqref{EQadjVI} one computes
\begin{equation}\label{eq:zsharpVI}
\bigl(f(z_{1},z_{2},z_{3})\bigr)^{\sharp}
=\bigl(z_{2}z_{3},\,z_{1}z_{3},\,z_{1}z_{2},\,0,\,0,\,0\bigr).
\end{equation}
Substituting \eqref{eq:zsharpVI} into the generic norm \eqref{EQgennormVI}  gives
\begin{equation}\label{eq:gennormVI}
\begin{aligned}
N_{\Omega_{\mathrm{VI}}}\bigl(f(z),\overline{f(z)}\bigr)
&=1-\bigl(|z_{1}|^{2}+|z_{2}|^{2}+|z_{3}|^{2}\bigr)
+\bigl(|z_{2}z_{3}|^{2}+|z_{1}z_{3}|^{2}+|z_{1}z_{2}|^{2}\bigr)
-|z_{1}z_{2}z_{3}|^{2}\\
&=(1-|z_{1}|^{2})(1-|z_{2}|^{2})(1-|z_{3}|^{2}).
\end{aligned}
\end{equation}
Therefore
\[
f^{*}\!\left(-\frac{i}{2}\,\partial\bar\partial\log N_{\Omega_{\mathrm{VI}}}\right)
=-\frac{i}{2}\,\partial\bar\partial\log\!\Bigl(\prod_{j=1}^3(1-|z_j|^2)\Bigr)
=\omega_{\Delta^{3}},
\]
so $g_{\Delta^{3}}=f^{*}g_{\Omega_{\mathrm{VI}}}$. Finally, the subspace
\[
W=\bigl\{(z_{1},z_{2},z_{3},Z_{1},Z_{2},Z_{3})\in\mathcal{H}\ \big|\ Z_{1}=Z_{2}=Z_{3}=0\bigr\}
\]
is a sub–HPJTS of $\mathcal H$, hence $\Pi=\operatorname{Im}(f)=W\cap\Omega_{\mathrm{VI}}$ is totally geodesic by the standard correspondence between sub–HPJTS and totally geodesic submanifolds (see \cite[Prop.~2.1]{DISCALALOI2008sympdual}).  Since the domain of type \(VI\) has rank \(3\), \(\Pi\) is of maximal rank.
This completes the proof.
\end{proof}


\subsection{The exceptional domain of type V}
Let \(\bigl(\mathcal{H}',\{\cdot,\cdot,\cdot\}'\bigr)\) be the sub–HPJTS of \(\bigl(\mathcal{H},\{\cdot,\cdot,\cdot\}\bigr)\) defined by
\[
\mathcal{H}'
\;:=\;\bigl\{(0,0,0,0,Z_{2},Z_{3})\in\mathcal{H}\ \big|\ Z_{2},Z_{3}\in\mathbb{O}_{\mathbb{C}}\bigr\},
\]
where \(\{\cdot,\cdot,\cdot\}'\) denotes the triple product induced from \(\{\cdot,\cdot,\cdot\}\) via \eqref{EQtriplPrVI}.

The \emph{exceptional domain of type V} is the \(16\)-dimensional bounded symmetric domain
\[
\Omega_{\mathrm{V}}
\;=\;\Omega_{\mathrm{VI}}\cap\mathcal{H}'
\;=\;
\Bigl\{\,z\in\mathcal{H}'\ \Big|\ 
1-(z\mid  z)+(z^{\sharp}\mid {z^{\sharp}})>0,\quad
2-(z\mid z)>0
\Bigr\},
\]
where the second equality follows from the defining inequalities of \(\Omega_{\mathrm{VI}}\) in \eqref{EQeqVI}.

Its associated generic norm (two–point form) is
\begin{equation}\label{EQgennormV}
N_{\Omega_{\mathrm{V}}}(z,\ov w)
\;=\;
N_{\Omega_{\mathrm{VI}}}(z,\ov w)
\;=\;
1-(z\mid w)+(z^{\sharp}\mid w^{\sharp}),
\qquad z,w\in \mathcal{H}',
\end{equation}
since \((z^{\sharp}\mid\ov z)=0\) for all \(z\in\mathcal{H}'\).

In order to exhibit a totally geodesic maximal–rank polydisk inside the exceptional domain of type \(V\), we next construct an explicit holomorphic embedding of \(\Delta^2\) into \(\Omega_{\mathrm{V}}\) and verify that it is both isometric and totally geodesic.


\begin{prop}\label{PROPpolymaxV}
The holomorphic map
\[
f \colon \Delta^2 \longrightarrow \Omega_{\mathrm{V}},\qquad
f(z_{1},z_{2})
=\bigl(0,\,0,\,0,\,0,\;\tfrac{z_{1}-z_{2}}{\sqrt{2}}+ i\,\tfrac{z_{1}+z_{2}}{\sqrt{2}}\,e_{1},\;0\bigr),
\]
is an embedding and satisfies
\[
g_{\Delta^2}=f^{*}g_{\Omega_{\mathrm{V}}}.
\]
In particular, its image
\[
\Pi'=\operatorname{Im}(f)
=\bigl\{(0,0,0,0,Z_{2},0)\in\Omega_{\mathrm{V}}\ \big|\ Z_{2}\in\operatorname {Span}_{\mathbb C}\{1,e_{1}\}\bigr\}
\]
is a totally geodesic polydisk of maximal rank in \(\Omega_{\mathrm{V}}\).
\end{prop}

\begin{proof}
Write
\[
Z_{2}=\alpha+\beta\,e_{1},
\qquad
\alpha=\tfrac{z_{1}-z_{2}}{\sqrt{2}},\quad
\beta=i\,\tfrac{z_{1}+z_{2}}{\sqrt{2}},
\]
so that \(f(z_{1},z_{2})=(0,0,0,0,Z_{2},0)\in\mathcal H'\).
Using the adjoint formula \eqref{EQadjVI} (with \(z_{1}=z_{2}=z_{3}=0\) and \(Z_{1}=Z_{3}=0\)) one obtains
\[
\bigl(f(z_{1},z_{2})\bigr)^{\sharp}
=\bigl(0,\;-\tfrac12\langle Z_{2},Z_{2}\rangle,\;0,\;0,0,0\bigr).
\]
Since \(\langle Z_{2},Z_{2}\rangle=-2\,z_{1}z_{2}\),
it follows that
\[
\bigl(f(z_{1},z_{2})\bigr)^{\sharp}=(0,\;z_{1}z_{2},\;0,\;0,0,0).
\]
By \eqref{EQgennormV}  the generic norm of \(\Omega_{\mathrm V}\) along the image of \(f\) reduces to
\[
N_{\Omega_{\mathrm V}}\bigl(f(z),\overline{f(z)}\bigr)
=1-(f(z)\mid f(z))+\bigl(f(z)^{\sharp}\mid f(z)^{\sharp}\bigr).
\]
Here,
\[
(f(z)\mid f(z))=|\alpha|^{2}+|\beta|^{2}
=\tfrac12\bigl(|z_{1}-z_{2}|^{2}+|z_{1}+z_{2}|^{2}\bigr)=|z_{1}|^{2}+|z_{2}|^{2},
\]
and \(\bigl(f(z)^{\sharp}\mid f(z)^{\sharp}\bigr)=|z_{1}z_{2}|^{2}\).
Therefore
\begin{equation}\label{eq:gennormV}
N_{\Omega_{\mathrm V}}\bigl(f(z),\overline{f(z)}\bigr)
=1-|z_{1}|^{2}-|z_{2}|^{2}+|z_{1}z_{2}|^{2}
=(1-|z_{1}|^{2})(1-|z_{2}|^{2}).
\end{equation}
We obtain
\[
f^{*}\omega_{\Omega_{\mathrm V}}
=-\frac{i}{2}\,\partial\bar\partial
\log\bigl((1-|z_{1}|^{2})(1-|z_{2}|^{2})\bigr)
=\omega_{\Delta^{2}},
\]
so \(g_{\Delta^{2}}=f^{*}g_{\Omega_{\mathrm V}}\) and \(f\) is an isometric holomorphic immersion.

The map \(f\) is linear and injective; moreover, the inverse on its image is holomorphic:
if \(Z_{2}=\alpha+\beta e_{1}\) with \(\alpha,\beta\in\mathbb C\), then
\[
z_{1}=\frac{\alpha-i\beta}{\sqrt{2}},
\qquad
z_{2}=-\frac{\alpha+i\beta}{\sqrt{2}}.
\]
Hence \(f\) is an embedding.
Finally, consider the sub–HPJTS of \(\mathcal{H}'\)
\[
W'=\bigl\{(0,0,0,0,Z_{2},0)\in\mathcal H'\ \big|\ Z_{2}\in\Span_{\mathbb C}\{1,e_{1}\}\bigr\}.
\]
By the bijection between sub–HPJTS and totally geodesic Hermitian symmetric subdomains (cf.\ \cite[Proposition 2.1]{DISCALALOI2008sympdual}), it follows that
\[
\Pi' \;=\;\operatorname{Im}(f)
\;=\;W'\cap\Omega_{\mathrm{V}}
\]
is a totally geodesic, holomorphically isometric embedding of \((\Delta^{2},g_{\Delta^{2}})\) inside \(\Omega_{\mathrm{V}}\).  Since the domain of type \(V\) has rank \(2\), \(\Pi'\) is of maximal rank.
\end{proof}


\section{Proofs of the main results}

\subsection{Proof of Theorem \ref{THMhpoly}}
\noindent
\noindent
In Lemma~\ref{lemtotgroos} below, together with the immediately following Proposition~\ref{PROPPolyDisk1}, we prove the theorem at the base point \(p=(0,0)\) for a suitably chosen tangent vector \(X\in T_pM_{\Omega,\mu}\).  In the subsequent lemmas, we exploit the domain’s symmetries to extend this result to arbitrary points and tangent directions, thus completing the proof of Theorem \ref{THMhpoly}.
\begin{lemma}\label{lemtotgroos}
Let $\Omega\subset\mathbb{C}^n$ be a bounded symmetric domain of rank $r$, realized in its circular embedding $\W\subset\mathbb{C}^n$. Suppose $\Omega$ decomposes as a product of irreducible Cartan domains, $\Omega\simeq\Omega_1\times\cdots\times\Omega_t$, with ranks $r_\ell$ so that $r=r_1+\cdots+r_t$.
For each $\ell=1,\dots,t$, let
$$
f_\ell:\; (\Delta^{r_\ell},\,g_{\Delta^{r_\ell}})\longrightarrow (\W_\ell,\,g_{\Omega_\ell}),\  f(0)=0,
$$
be the  totally geodesic holomorphic isometry constructed in Propositions~\ref{PROPpolymaxVI} and \ref{PROPpolymaxV} for exceptional domains, and in \cite[Sec.~2]{MOSSAZEDDA2022polch} for classical  Cartan domains. 
Define
\begin{equation}\label{eq:fnsdjfgns2}
f\;=\;f_1\times\cdots\times f_t:\;(\Delta^{r},\,g_{\Delta^{r}})\longrightarrow (\W,\,g_{\Omega})
\end{equation}
and set $\Pi:= f(\Delta^r)\subset\W$.
Then, there exists a permutation of the standard coordinates $(z_1,\dots,z_n)$ on $\mathbb{C}^n$ such that
\[
\Pi \;=\; \bigl\{\,z\in\Omega \;\big|\; z_{r+1}=\cdots=z_{n}=0\,\bigr\}.
\]

Moreover, let $M_{\Omega,\mu}$ be the Hartogs-type domain over $\Omega$ endowed with its Kähler metric $g:= g_{{\Omega,\mu}}$, and write local holomorphic coordinates on $M_{\Omega,\mu}$ as $(z_0,z_1,\dots,z_n)$, where $z_0\in\mathbb{C}$ is the fiber coordinate and $z=(z_1,\dots,z_n)\in\W$ is the base coordinate. Denote
\[
g_{j\bar k}\;=\; g\!\bigl(\partial_{z_j},\,\partial_{\bar z_k}\bigr),\qquad 0\le j,k\le n.
\]
Then, along the complex submanifold $\{(z_0,z)\in\mathbb{C}\times\Pi\}\subset M_{\Omega,\mu}$, the Hermitian matrix $\bigl(g_{j\bar k}(z_0,z)\bigr)_{0\le j,k\le n}$ is block–diagonal:
\begin{equation}\label{eqgjk=0roos}
\bigl(g_{j\bar k}(z_{0},z)\bigr)\;=\;
\begin{pmatrix}
\bigl(g_{j\bar k}\bigr)_{0\le j,k\le r} & \;0\\[6pt]
0 & \bigl(g_{j\bar k}\bigr)_{\,r<j,k\le n}
\end{pmatrix},
\qquad (z_0,z)\in\mathbb{C}\times\Pi,
\end{equation}
i.e., $g_{j\bar k}(z_0,z)=0$ whenever $0\le j\le r<k\le n$ or $0\le k\le r<j\le n$.
\end{lemma}

\begin{proof}
Clearly, it suffices to prove the lemma in the case where $\Omega$ is irreducible (equivalently, $t=1$), so that its circular realization $\W$ is a single Cartan domain. The classical types are treated in \cite{MOSSAZEDDA2022polch}; therefore we restrict to the two exceptional cases (types V and VI).

In order to prove \eqref{eqgjk=0roos}, observe that on any Hartogs domain \(M_{\Omega,\mu}\) the Calabi diastasis at the origin for the metric \(g=g_{{\Omega,\mu}}\) is  given by
\begin{equation}\label{EQdistagmw}
D_0^{g}(z,w)
\;=\;
-\log\bigl(N_{\Omega}(z,\bar z)^\mu - |z_0|^2\bigr)=
F\bigl(N_{\Omega}(z,\bar z),\,|z_0|^2\bigr),
\end{equation}
where \(F(x,y)=-\log(x^\mu-y)\).
Differentiating twice gives, for $j,k\ge1$,
\begin{equation}\label{eqdedebD1}
\begin{split}
g\bigl(\partial_{z_j},\partial_{\bar z_k}\bigr)
&=\frac{\partial^2D^g_0}{\partial z_j\,\partial\bar z_k}
=\frac{\partial^2N}{\partial z_j\partial\bar z_k}\,F_x
+\frac{\partial N}{\partial z_j}\,\frac{\partial N}{\partial\bar z_k}\,F_{xx},
\end{split}
\end{equation}
and for the mixed component with $z_0$,
\begin{equation}\label{eqdedebD2}
\begin{split}
g\bigl(\partial_{z_0},\partial_{\bar z_k}\bigr)
&=\frac{\partial^2D^g_0}{\partial z_0\,\partial\bar z_k}
=\bar z_0\,\frac{\partial N}{\partial\bar z_k}\,F_{xy}.
\end{split}
\end{equation}
Here $N=N_\Omega(z,\bar z)$ and subscripts on $F$ denote partial derivatives.
Hence to prove the block‐diagonal form \eqref{eqgjk=0roos} it suffices to show
\begin{equation}\label{EQdeNWzk}
\frac{\partial N}{\partial\bar z_k}(z,\bar z)=0,
\quad
\text{and}
\quad
\frac{\partial^2N}{\partial z_j\,\partial\bar z_k}(z,\bar z)=0,
\end{equation}
for $z \in \Pi$ and $0\leq j \leq r < k \leq n$.
 
Let $\Omega=\Omega_{VI}$ and let
\[
\Pi=\{(z_1,z_2,z_3,0, 0,0)\}\subset\Omega_{VI}
\]
be the polydisk of rank $3$ from Proposition~\ref{PROPpolymaxVI}.  We must verify that  the following hold for all \(j=0,1,2,3\), \(s=1,2,3\), \(t=1,\dots,8\):  
\begin{equation}\label{EQde2NWVI_revised}
\begin{aligned}
0
&=\frac{\partial^2 N_{\Omega_{\mathrm{VI}}}}{\partial z_j\,\partial\bar z_{st}}
\Big\vert_{\Delta^3} \\[-0.3em]
&=\Bigl[
  -\bigl(\tfrac{\partial z}{\partial z_j}\mid\tfrac{\partial z}{\partial z_{st}}\bigr)
  +\bigl(\tfrac{\partial z^\sharp}{\partial z_j}\mid\tfrac{\partial{z}^\sharp}{\partial z_{st}}\bigr) \\[-0.3em]
&\quad\;\;
  -\tfrac19\bigl(\,( \tfrac{\partial z^\sharp}{\partial z_j}\mid\bar z)
        + (z^\sharp\mid\tfrac{\partial\bar z}{\partial\bar z_j})\bigr)
    \bigl(\,( \tfrac{\partial\bar{z}^\sharp}{\partial\bar z_{st}}\mid  z)
        +(\bar z^\sharp\mid\tfrac{\partial z}{\partial z_{st}})\bigr)
\Bigr]_{\Delta^3},
\end{aligned}
\end{equation}
and
\begin{equation}\label{EQdeNWVI_revised}
\begin{aligned}
0
&=\frac{\partial N_{\Omega_{\mathrm{VI}}}}{\partial\bar z_{st}}
\Big\vert_{\Delta^3} \\[-0.3em]
&=\Bigl[
  -\bigl(z\mid\tfrac{\partial z}{\partial z_{st}}\bigr)
  +\bigl(z^\sharp\mid\tfrac{\partial{z}^\sharp}{\partial z_{st}}\bigr) \\[-0.3em]
&\quad\;\;
  -\tfrac19\,(z^\sharp\mid\bar z)\,
    \bigl(\,( \tfrac{\partial\bar{z}^\sharp}{\partial\bar z_{st}}\mid z)
        +(\bar z^\sharp\mid\tfrac{\partial z}{\partial z_{st}})\bigr)
\Bigr]_{\Delta^3}, 
\end{aligned}
\end{equation}
where
$Z_s=z_{s0}+\sum_{j=1}^{7}z_{sj} e_j, s=1, 2, 3$ (cf. \eqref{HZ123}).
 Both identities follow directly from the adjoint formula \eqref{eq:zsharpVI},
together with the explicit derivatives
\begin{equation}\label{EQdebst_revised}
\frac{\partial\overline{z^\sharp}}{\partial\bar z_{st}}\Big\vert_{\Delta^3}
=\bigl(0,0,0,\; -\overline{z}_1\,\varepsilon_t\,\delta_{1s},\; -\overline{z}_2\,\varepsilon_t\,\delta_{2s},\; -\overline{z}_3\,\varepsilon_t\,\delta_{3s}\bigr),
\end{equation}
\begin{equation}\label{EQdej_revised}
\frac{\partial z^\sharp}{\partial z_j}\Big\vert_{\Delta^3}
=\bigl(\delta_{j2}z_3+\delta_{j3}z_2,\;\delta_{j3}z_1+\delta_{j1}z_3,\;\delta_{j1}z_2+\delta_{j2}z_1,\,0,0,0\bigr).
\end{equation}
proving the required vanishing of the mixed derivatives
\eqref{EQdeNWzk}.

Let \(\Omega = \Omega_{V}\).  Consider the coordinates \((z_0,z) = (z_0, Z_2, Z_3)\) on \(\mathbb{C} \times \mathbb{O}_{\mathbb{C}}^2 \cong \C \times \mathcal H '\), restricted to the Hartogs domain
\[
M_{\Omega_V,\mu}
=\bigl\{(z_0,z)\in\C\times\Omega_V \mid |z_0|^2 < N_{\Omega_V}(z,\bar z)\bigr\}.
\]
Let \(\Delta^2\subset\Omega_V\) be the totally geodesic polydisk furnished by Proposition~\ref{PROPpolymaxV}.  We must check that for all multi‐indices \(st = 2t\) with \(t=3,\dots,8\), \(st=3t\) with \(t=1,\dots,8\), and \(j=1,2\), the following vanish on \(\Delta^2\):

\begin{align}
0&=\frac{\partial N_{\Omega_V}}{\partial\bar z_{st}}\Big\vert_{\Delta^2}
=\Bigl[-\bigl(z\mid\tfrac{\partial z}{\partial z_{st}}\bigr)
       +\bigl(z^\sharp\mid\tfrac{\partial{z}^\sharp}{\partial z_{st}}\bigr)\Bigr]_{\Delta^2},
\label{EQdeNWV_revised}\\
0&=\frac{\partial^2N_{\Omega_V}}{\partial z_{2j}\,\partial\bar z_{st}}\Big\vert_{\Delta^2}
=\Bigl[-\bigl(\tfrac{\partial z}{\partial z_{2j}}\mid\tfrac{\partial z}{\partial z_{st}}\bigr)
       +\bigl(\tfrac{\partial z^\sharp}{\partial z_{2j}}\mid\tfrac{\partial{z}^\sharp}{\partial z_{st}}\bigr)\Bigr]_{\Delta^2}.
\label{EQde2NWV_revised}
\end{align}

To see this, we need the following computations:
from the adjoint formula \eqref{EQadjVI}, the restriction of \(z^\sharp\) is
\begin{equation}\label{EQdies}
z^\sharp\big\vert_{\Delta^2}
=\bigl(0,\,-\tfrac12\,(z_{21}^2+z_{22}^2),\,0,\dots,0\bigr).
\end{equation}
differentiating \(\overline{z^\sharp}\) with respect to \(\bar z_{st}\) gives
\begin{equation}\label{EQdebstV_revised}
\frac{\partial\overline{z^\sharp}}{\partial\bar z_{st}}\Big\vert_{\Delta^2}
=\bigl(0,0,0,\,\partial_{\bar z_{st}}(\overline{Z}_2\,\overline{Z}_3),\,0,0\bigr).
\end{equation}
 differentiating \(z^\sharp\) with respect to \(z_{2j}\) yields
\begin{equation}\label{EQdejV_revised}
\frac{\partial z^\sharp}{\partial z_{2j}}\Big\vert_{\Delta^2}
=\bigl(0,\,-z_{2j},\,0,\,0,0,0\bigr).
\end{equation}  

Substituting \eqref{EQdies}, \eqref{EQdebstV_revised} and \eqref{EQdejV_revised} into \eqref{EQdeNWV_revised} and \eqref{EQde2NWV_revised} shows the required vanishing of the mixed derivatives. The proof is complete.

\end{proof}

\begin{prop}\label{PROPPolyDisk1}
Let $M_{\Omega,\mu}$ be a Hartogs domain over a  bounded symmetric domain $\Omega$ of rank $r$, and let 
$\Pi\subset\Omega$
be the totally geodesic polydisk provided by Lemma~\ref{lemtotgroos}.  Then the subset
$(\C\times\Pi)\,\cap\,M_{\Omega,\mu}$ is biholomorphic to the Hartogs domain $M_{\Delta^r,\mu}$ and it is a totally geodesic Kähler submanifold of $(M_{\Omega,\mu},g_{{\Omega,\mu}})$.
\end{prop}

\begin{proof}
Let $\W\subset\C^n$ be a bounded symmetric  domain, and let $(z_0,z_1,\dots,z_n)$ denote the coordinates on $\C^n$, with indices ordered as in Lemma \ref{lemtotgroos}. From Lemma \ref{lemtotgroos} we see that, along the totally geodesic polydisk
$$
\Pi \;=\;\{\,z_{r+1}=\cdots=z_n=0\,\}\;\subset\;\W,
$$
the metric components satisfy
$$
g^{j\bar k}(z)=0
\quad\text{and}\quad
\frac{\partial\,g_{j\bar k}}{\partial z_\ell}(z)=0
\quad
\text{whenever } j,\ell\le r<k.
$$
Accordingly, for any indices $\ell,j\le r<t$, the Christoffel symbols
$$
\Gamma_{\ell j}^t(z)
=\sum_{k=0}^n g^{t\bar k}(z)\,\frac{\partial\,g_{j\bar k}}{\partial z_\ell}(z)
$$
vanish identically on $\{z_{r+1}=\cdots=z_n=0\}$.  It follows that the second fundamental form of the inclusion
$$
(\C\times\Pi)\cap M_{\Omega,\mu}\;\hookrightarrow\;M_{\Omega,\mu}
$$
is identically zero, and hence this inclusion is totally geodesic.
By  Proposition \ref{PROPpolymaxVI}\eqref{eq:gennormVI} and Proposition \ref{PROPpolymaxV}\eqref{eq:gennormV}
the  the generic norm on \(\Pi\), satisfies
\[
N_{\Omega}\bigl(f(z),\overline{f(z)}\bigr)
\;=\;\prod_{i=1}^r \bigl(1 - |z_i|^2\bigr)
\;=\;N_{\Delta^r}(z,\bar z),
\]
where $f$ is given by \eqref{eq:fnsdjfgns2}.
Hence $\mathrm{id}_{\C}\times f: M_{\Delta^r,\mu}\rightarrow  (\C\times\Pi)\,\cap\,M_{\Omega,\mu}$
is a biholomorphism and the map  
\[
M_{\Delta^r,\mu}
\;\xrightarrow{\;\mathrm{id}_{\C}\times f\;}
(\C\times\Pi)\,\cap\,M_{\Omega,\mu}
\;\hookrightarrow\;M_{\Omega,\mu}
\]
is a totally geodesic \K\ embedding.
\end{proof}

\noindent
In the following lemma we show that every automorphism of the base domain naturally lifts to an automorphism of the associated Hartogs domain.
\begin{lem}\label{LEMliftgen}
Let $\Omega$ be a bounded symmetric domain and let $\phi\colon \Omega\to\Omega$ be an isometric automorphism.  Then $\phi$ lifts to a biholomorphism 
\[
\tilde\phi\colon M_{\Omega, \mu}\;\longrightarrow\;M_{\Omega, \mu},
\]
defined by
\begin{equation}\label{EQliftgen}
\tilde{\phi}(w, z)
=\bigl(e^{\mu\,h_{\phi}(z)}\,w, \phi(z)\bigr),
\end{equation}
where $h_{\phi}\colon\Omega\to\C$ is a holomorphic function chosen so that
\begin{equation}\label{EQliftgen2}
\tilde\phi^*g_{{\Omega,\mu}} = g_{{\Omega,\mu}}
\end{equation}
Moreover, if $\phi$ fixes the origin, one may take $h_{\phi}\equiv0$, so that
\[
\tilde{\phi}(w, z)=(w, \phi(z)).
\]
\end{lem}

\begin{proof}
Let $\phi\colon \Omega \to \Omega$ be an isometric automorphism.  Since $\phi$ preserves the Kähler form associated to the generic norm $N_\Omega$, we have
\[
\partial\bar\partial\log N_\Omega\bigl(\phi(z),\overline{\phi(z)}\bigr)
=\partial\bar\partial\log N_\Omega(z,\bar z).
\]
Hence, there exists a holomorphic function $h_\phi\colon \Omega\to\C$ such that
\[
N_\Omega\bigl(\phi(z),\overline{\phi(z)}\bigr)
= N_\Omega(z,\bar z)\,\exp\bigl(h_\phi(z)+\overline{h_\phi(z)}\bigr).
\]
It follows that the map
\[
\tilde\phi\colon M_{\Omega, \mu}\longrightarrow M_{\Omega, \mu},
\quad
\tilde\phi(w, z) 
= \bigl(e^{\mu\,h_\phi(z)}\,w, \phi(z)\bigr),
\]
is well defined, since
\[
\bigl|e^{\mu\,h_\phi(z)}w\bigr|^2
< \bigl|e^{\mu\,h_\phi(z)}\bigr|^2\,N^\mu_\Omega(z,\bar z)
= N^\mu_\Omega\bigl(\phi(z),\overline{\phi(z)}\bigr).
\]

The identity  in \eqref{EQliftgen2} follows from the fact that  Calabi's diastasis functions are invariant under the lift $\tilde\phi$.  Indeed, one checks directly that:
\begin{align}
\partial\bar\partial\,D^{g_{{\Omega,\mu}}}_0\bigl(\tilde\phi(w, z)\bigr)
&=\partial\bar\partial\log\bigl(N_\Omega(\phi(z),\overline{\phi(z)})^\mu -|e^{\mu h_\phi(z)}w|^2\bigr)\notag\\
&=\partial\bar\partial\log\bigl(N_\Omega(z,\bar z)^\mu -|w|^2\bigr)
=\partial\bar\partial\,D^{g_{{\Omega,\mu}}}_0(w, z).
\label{EQlemphipb1}
\end{align}
Hence  the diastasis is preserved  and \eqref{EQliftgen2} follows.
Finally, since any biholomorphism \(\phi\) of \(\Omega\) fixing the origin preserves the corresponding positive Hermitian Jordan triple structure (see, e.g., \cite[p.~551]{UPMEIER1984Toeplitzbsd}), it follows that the generic norm \(N_{\Omega}\) is invariant under \(\phi\).  In this case, one can take \(h_{\phi}\equiv0\) in \eqref{EQliftgen}, yielding the simpler lift
$\tilde\phi(w, z) \;=\; \bigl(w, \phi(z)\bigr)$.
\end{proof}

Here, by combining the preceding results, we show that a totally geodesic polydisk of maximal rank in the base domain gives rise to a totally geodesic Hartogs–polydisk in the corresponding Hartogs domain.
\begin{prop}\label{PROPPolyDisk2}
Let $\Omega$ be a bounded symmetric domain of  rank $r$, and let 
$\Pi\subset\Omega$
be a totally geodesic polydisk of complex dimension $r$.  Define
\[
C_{\Pi}
=\bigl(\C\times\Pi\bigr)\,\cap\,M_{\Omega,\mu}
=\{(w, z)\in M_{\Omega,\mu}\mid z\in\Pi\}.
\]
Then $C_{\Pi}\simeq M_{\Delta^r, \mu}$ and it  is a totally geodesic Kähler submanifold of $(M_{\Omega,\mu}, g_{{\Omega,\mu}}).$ 
\end{prop}

\begin{proof}
We may assume without loss of generality that the chosen polydisk $\Pi\subset\Omega$ passes through the origin.  The conclusion then follows from Proposition~\ref{PROPPolyDisk1} together with the Polydisk Theorem (see \cite{WolfBook1969FineStructureHSSpolydisc},  \cite[Thm.~VI.3.5]{ROOSinFarautKoranyiRoosBook2000}), which asserts that the identity component $\Aut_0(\Omega)$ acts transitively on the collection of all $r$-dimensional totally geodesic polydisks through the origin.  Applying to $M_{\Omega,\mu}$ the lifts of these automorphisms, as described in Lemma \ref{LEMliftgen}, completes the proof.
\end{proof}

By combining the preceding proposition with the classical Polydisk Theorem for Hermitian symmetric spaces, we complete the proof of Theorem~\ref{THMhpoly}.

\begin{proof}[Proof of Theorem \ref{THMhpoly}]
Let $p=(w, z)\in M_{\Omega,\mu}$ and fix a tangent vector
$$
X\in T_{p}M_{\Omega,\mu}.
$$
Write $X = X_1 + X_2$ with $X_1\in T_z\Omega$ and $X_2\in T_w\C$.  By the Polydisk Theorem (\cite{WolfBook1969FineStructureHSSpolydisc}), there exists a totally geodesic polydisk
$$
j\colon (\Delta^r, g_{\Delta^r}) \rightarrow (\Omega, g_{\Omega})
$$
of rank $r$ passing through $z$ such that $X_1\in T_z\Pi$, where $\Pi = \operatorname{Im}(j)$.  
Noticing that
$X\in T_{(w,z)}C_{\Pi}$
we conclude the proof by taking as 
$$\widetilde{j}:M_{\Delta^r, \mu}\simeq C_{\Pi}\hookrightarrow M_{\Omega,\mu}$$
the inclusion provided by Proposition \ref{PROPPolyDisk2}. 
\end{proof}


\subsection{Proof of Theorem \ref{THMhpolydual}}
The proof follows the same outline as in the noncompact dual case.  First, we establish the existence of a polydisk with the desired properties (Proposition~\ref{PROPPolyDual}).  Next, we invoke the domain’s symmetries (Lemma~\ref{LEMliftgend}) and apply the classical Polydisk Theorem.  Finally, Proposition~\ref{PROPPolydual2} together with the preceding arguments completes the proof.

\vskip0.1cm

\noindent

Consider a bounded symmetric domain $\W=\W_1\times\cdots\times\W_t\subset\C^n$  where 
$\Omega_\ell\subset\C^{n_\ell}$ is a Cartan domain  of rank $r_{\ell}$ endowed with its hyperbolic metric $g_{\Omega_\ell}$ (cf.  Section~\ref{SECcartan}). On $\C^{n_\ell}$, introduce the dual metric $g_{\Omega_\ell}^*$ whose Kähler form is  
\[
\omega_{{\Omega_\ell}}^* \;=\; \frac{i}{2}\,\partial\bar\partial \log N_{\Omega_\ell}(z,-\bar z),
\]
and on $\C^{r_\ell}$ introduce the dual metric $g_{\Delta^{r_\ell}}^*$ with Kähler form  
\[
\omega_{{\Delta^{r_\ell}}}^* \;=\; \frac{i}{2}\,\partial\bar\partial \log\!\Bigl(\prod_{j=1}^{r_\ell} (1+|z_j|^2)\Bigr).
\]
In particular (see, e.g., \cite{MOSSAZEDDAS2022symplch}), $(\C^{n_\ell},g_{\Omega_\ell}^*)$ is an open dense chart of the compact dual Hermitian symmetric space associated with $(\Omega_\ell,g_{\Omega_\ell})$, and 
\[
(\C^{r_\ell},g_{\Delta^{r_\ell}}^*) \;\cong\; (\C,\omega_{\mathrm{FS}})^{r_\ell},
\]
where $\omega_{\mathrm{FS}}$ denotes the Fubini–Study metric on $\mathbb{C}P^1$ restricted to the standard affine coordinate chart. 

Let  
\begin{equation}\label{EQfdeltar}
f_\ell\colon (\Delta^{r_\ell},g_{\Delta^{r_\ell}}) \longrightarrow (\Omega_\ell,g_{\Omega_\ell})
\end{equation}
be one of the totally geodesic, holomorphic isometric immersions constructed in Propositions~\ref{PROPpolymaxVI} and~\ref{PROPpolymaxV} for exceptional domains, or in \cite[Sec.~2]{MOSSAZEDDA2022polch} for classical  Cartan domains. Since $f_\ell$ preserves the Calabi diastasis, we have
\[
D^{g_{\Omega_\ell}}\bigl(f_\ell(z),\overline{f_\ell(z)}\bigr)
=-\log N_{\Omega_\ell}\bigl(f_\ell(z),\overline{f_\ell(z)}\bigr)
=-\log\!\Bigl(\prod_{k=1}^{r_\ell}(1-|z_k|^2)\Bigr)
=D^{g_{\Delta^{r_\ell}}}(z,\bar z).
\]
By a slight abuse of notation, we also denote by $f_\ell\colon \C^{r_\ell}\to\C^{n_\ell}$ the natural linear extension of $f_\ell$. Inspecting the explicit formula for $f_\ell$ shows immediately that $f_\ell(-z)=-f_\ell(z)$. It follows that
\[
D^{g_{\Omega_\ell}^*}\bigl(f_\ell(z),\overline{f_\ell(z)}\bigr)
=\log N_{\Omega_\ell}\bigl(f_\ell(z),-\overline{f_\ell(z)}\bigr)
=\log\!\Bigl(\prod_{k=1}^{r_\ell}(1+|z_k|^2)\Bigr)
=D^{g_{\Delta^{r_\ell}}^*}(z,\bar z).
\]
We conclude that
\[
f_\ell\colon (\C^{r_\ell},\,g_{\Delta^{r_\ell}}^*) \longrightarrow (\C^{n_\ell},\,g_{\Omega_\ell}^*)
\]
is a holomorphic isometric embedding.

Consider now the holomorphic isometric embedding
\[
f:=f_1\times\cdots\times f_t:\;(\C^{r},\,g_{\Delta^{r}}^*) \longrightarrow (\C^{n},\,g_{\Omega}^*),
\qquad
g_{\Omega}^*=g_{\Omega_1}^*\times\cdots\times g_{\Omega_t}^*.
\]
We will show that $f$ is also totally geodesic. Adopting the index ordering of Lemma~\ref{lemtotgroos}, define the linear subspace
\begin{equation}\label{EQstdrPidual}
\Pi^* \;=\; \{\,z\in\C^n \mid z_{r+1}=\cdots=z_n=0\}.
\end{equation}
Let $g^*:=g^*_{{\Omega,\mu}}$, and write
\[
g^*_{j\bar k}
= g^*\!\bigl(\tfrac{\partial}{\partial z_j},\,\tfrac{\partial}{\partial\bar z_k}\bigr),
\quad
j,k = 0,1,\dots,n.
\]
Since
\[
\partial\bar\partial D^{g^*}_0(z,\bar z)
= -\,\partial\bar\partial D^g_0(z,-\bar z),
\]
it follows from \eqref{eqgjk=0roos} that, along $(\C\times\Pi^*)\subset\C^{n+1}$, the Hermitian matrix $\bigl(g^*_{j\bar k}(z_0,z)\bigr)_{0\le j,k\le n}$ splits as
\begin{equation}\label{eqgjk=0dual}
\bigl(g^*_{j\bar k}(z_0,z)\bigr)_{0\le j,k\le n}
=
\begin{pmatrix}
\bigl(g^*_{j\bar k}\bigr)_{0\le j,k\le r} & 0\\[6pt]
0 & \bigl(g^*_{j\bar k}\bigr)_{r<j,k\le n}
\end{pmatrix}.
\end{equation}
By the same argument used in the proof of Proposition~\ref{PROPPolyDisk1}, one concludes that the submanifold
$\C \times \Pi^* \;\subset\; \C^{n+1}$
is totally geodesic. Moreover, the map
\[
\mathrm{id}\times f \colon \C^{r+1}\;\xrightarrow{\simeq}\;\C\times\Pi^*\;\hookrightarrow\;\C^{n+1}
\]
is a holomorphic isometry. We therefore obtain the following proposition.

\begin{prop}\label{PROPPolyDual}
Let $\Omega\subset\C^n$ be a bounded symmetric domain, realized in its circular form, of rank $r$, and let $\Pi^*\subset\C^n$ be the linear subspace defined in \eqref{EQstdrPidual}. Then 
$\C \times \Pi^* \;\subset\; \C^{n+1}$
is a totally geodesic Kähler submanifold of $\bigl(\C^{n+1},g^*_{{\Omega,\mu}}\bigr)$.
\end{prop}

\begin{lemma}\label{LEMliftgend}
Let \(\Omega\) be a bounded symmetric domain, realized in its circular form, and let 
\(\phi\colon (\C^n,g_{\Omega}^*)\to(\C^n,g_{\Omega}^*)\)
be an isometric biholomorphism fixing the origin.  Then the  map 
\[
\tilde\phi\colon (\C^{n+1},g^*_{{\Omega,\mu}})\longrightarrow(\C^{n+1},g^*_{{\Omega,\mu}})
\]
given by $\tilde\phi(z_0, z) \;=\; \bigl(z_0, \phi(z)\bigr)$ is an isometric biholomorphism.
\end{lemma}

\begin{proof}
By the hereditary property of the Calabi diastasis, for all \(z\in\C^n\) one has
\[
D_0^{g_{\Omega}^*}\bigl(\phi(z)\bigr)
= D_0^{g_{\Omega}^*}(z)
\;\Longrightarrow\;
N_{\Omega}\bigl(\phi(z),-\overline{\phi(z)}\bigr)
= N_{\Omega}(z,-\bar z).
\]
Arguing as in the proof of Lemma~\ref{LEMliftgen}, consider the map
\[
\tilde\phi\colon \C^{n+1}\longrightarrow\C^{n+1},
\qquad
\tilde\phi(z_0,z) = \bigl(z_0,\phi(z)\bigr).
\]
One checks that \(\tilde\phi\) preserves the diastasis, namely
\[
D_0^{g^*_{{\Omega,\mu}}}(z_0,z)
= D_0^{g^*_{{\Omega,\mu}}}\bigl(\tilde\phi(z_0,z)\bigr),
\]
It follows that \(\tilde\phi\) is an isometric biholomorphism of \((\C^{n+1},g^*_{{\Omega,\mu}})\), as required.
\end{proof}

\begin{prop}\label{PROPPolydual2}
Let \(\Omega\) be a rank $r$ bounded symmetric domain, realized in its circular form, and let 
\(\Pi^*\subset\C^{n}\) be an \(r\)-dimensional dual polydisk through the origin which is totally geodesic with respect to the dual hyperbolic metric \(g_\Omega^*\).  Then the linear subspace
\[
C_{\Pi^*}
= \C\times\Pi^*\simeq \C^{r+1}
\;\subset\;\C^{n+1}
\]
is a totally geodesic Kähler submanifold of \(\bigl(\C^{n+1},g^*_{{\Omega,\mu}}\bigr)\).  
\end{prop}

\begin{proof}
By the Polydisk Theorem (see \cite{WolfBook1969FineStructureHSSpolydisc}), the isotropy subgroup at the origin in the group of holomorphic isometries of \((\C^n,g_{\Omega}^*)\) acts transitively on the family of \(r\)-dimensional, totally geodesic polydisks through the origin.  Combining this transitivity with Proposition \ref{PROPPolyDual} and Lemma \ref{LEMliftgend} yields the desired result.
\end{proof}

\subsubsection*{Proof of Theorem \ref{THMhpolydual}}
Let \(X\in T_{(z_0,0)}\C^{n+1}\) be an arbitrary tangent vector, and decompose \(X=X_1+X_2\) with \(X_1\in T_{z_0}\C\) and \(X_2\in T_0\C^n\).  
By the dual  Polydisk Theorem (\cite{WolfBook1969FineStructureHSSpolydisc}), there exists a totally geodesic  dual polydisk
$$
j\colon (\Pi^*, g^*_{\Delta^r}) \rightarrow (\Omega^*, g^*_{\Omega})
$$
of rank $r$ passing through the origin such that \(X_2\in T_0\Pi^*\).
Noticing that
\[
X\in T_{(z_0,0)}(\C\times\Pi^*) = T_{(z_0,0)}C_{\Pi^*},
\]
we conclude the proof by taking as 
$$\widetilde{j}:\C^{r+1}\simeq \C\times\Pi^*\rightarrow \C^{n+1}, (z_0, z)\mapsto (z_0, j(z))$$
the inclusion provided by Proposition \ref{PROPPolydual2}.


\subsection{Proof of Theorem \ref{THMCORhpolydual}}

By Theorem \ref{THMhpolydual}, for any bounded symmetric domain \(\W\) of rank \(r\), there exists a holomorphically totally geodesic inclusion
\[
\; \bigl(\C^{s+1},\,g^*_{\Delta^s,\mu}\bigr)\;\hookrightarrow\;\bigl(\C^{n+1},\,g^*_{\W,\mu}\bigr),
\]
where one takes \(s=2\) if \(r\ge2\), and \(s=1\) otherwise.  A straightforward computation of the curvature tensor shows that:
\begin{enumerate}
  \item The sectional curvature of $\bigl(\C^{2},\,g^*_{\Delta^1,\mu}\bigr)$ at \((w, 0)\) in the plane spanned by \(\frac{\partial}{\partial x_1}\) and \(\frac{\partial}{\partial y_1}\) is
  \begin{equation}\label{EQnocptmuneq1}
    K\Bigl(\tfrac{\partial}{\partial x_1}\wedge \tfrac{\partial}{\partial y_1}\Bigr)(w, 0)
    \;=\;
    \frac{2 + 2(\mu-1)\,\lvert w\rvert^2}{\mu},
  \end{equation}
  where \(\frac{\partial}{\partial z_1}=\tfrac12\bigl(\tfrac{\partial}{\partial x_1}-i\,\tfrac{\partial}{\partial y_1}\bigr)\).
  
  \item The sectional curvature of $\bigl(\C^{3},\,g^*_{\Delta^2,\mu}\bigr)$ at \((w, 0,0)\) in the plane spanned by \(\frac{\partial}{\partial x_1}\) and \(\frac{\partial}{\partial x_2}\) is
  \begin{equation}\label{EQnocptmuneq2}
    K\Bigl(\tfrac{\partial}{\partial x_1}\wedge \tfrac{\partial}{\partial x_2}\Bigr)(w, 0,0)
    \;=\;
    -\tfrac{\lvert w\rvert^2}{2},
  \end{equation}
  where \(\frac{\partial}{\partial z_j}=\tfrac12\bigl(\tfrac{\partial}{\partial x_j}-i\,\tfrac{\partial}{\partial y_j}\bigr)\) for \(j=1,2\).
\end{enumerate}

Since any compact Riemannian manifold has bounded sectional curvature, equation \eqref{EQnocptmuneq1} precludes \(\bigl(\C^{n+1},g^*_{\C H^n,\mu}\bigr)\) from being a totally geodesic submanifold of a compact manifold whenever \(\mu\neq1\).  Similarly, equation \eqref{EQnocptmuneq2} shows that if \(\rank(\W)\ge2\), then \(\bigl(\C^{n+1},g^*_{\W,\mu}\bigr)\) cannot be realized as a totally geodesic submanifold of any compact Riemannian manifold.  This completes the proof.


\end{document}